\documentclass[12pt]{article}
\usepackage{enumerate}
\usepackage{amssymb,amsmath,amsfonts,amsthm,mathrsfs}
\usepackage[colorlinks=true, pdfstartview=FitV, linkcolor=blue, citecolor=blue, urlcolor=blue]{hyperref}
\usepackage{graphicx,color}
\usepackage{cite}
\usepackage{indentfirst}

\allowdisplaybreaks

\begin{document}
\def\rn{{\mathbb R^n}}  \def\sn{{\mathbb S^{n-1}}}
\def\co{{\mathcal C_\Omega}}
\def\z{{\mathbb Z}}
\def\nm{{\mathbb (\rn)^m}}
\def\mm{{\mathbb (\rn)^{m+1}}}
\def\n{{\mathbb N}}
\def\cc{{\mathbb C}}

\newtheorem{defn}{Definition}
\newtheorem{thm}{Theorem}
\newtheorem{lem}{Lemma}
\newtheorem{cor}{Corollary}
\newtheorem{rem}{Remark}

\title{\bf\Large Sharp bounds on higher-dimensional product spaces for Hardy-type operators on Heisenberg group
\footnotetext{{\it Key words and phrases}: Hardy operator; weighted Hardy operator; weighted Ces$\mathrm{\grave{a}}$ro operator; higher-dimensional product space; Heisenberg group.
\newline\indent\hspace{1mm} {\it 2020 Mathematics Subject Classification}: Primary 42B25; Secondary 42B20, 47H60, 47B47.}}

\date{}
\author{Zhongci Hang, Wenfeng Liu,  Xiang Li\footnote{Corresponding author.} and Dunyan Yan}
\maketitle
\begin{center}
\begin{minipage}{13cm}
{\small {\bf Abstract:}\quad
 In this paper, we study sharp bound on higher-dimensional Lebesgue product space for Hardy operator on Heisenberg group, the constants of sharp bounds are  obtained. In addition, we also give the  boundedness for weighted Hardy operator and weighted Ces$\mathrm{\grave{a}}$ro operator.
 }
\end{minipage}
\end{center}

\section{Introduction}\label{sec1}
\par
The most fundamental averaging operator is Hardy operator defined by 
$$
\mathcal{H}f(x):=\frac{1}{x}\int_0^xf(t)dt,
$$
where the function $f$ is a nonnegative integrable function on $\mathbb{R}^+=(0,\infty)$ and $x>0$. A celebrated integral inequality, due to Hardy \cite{Hardy}, states that 
$$
\|\mathcal{H}(f)\|_{L^p(\mathbb{R}^+)}\leq\frac{p}{p-1}\|f\|_{L^p(\mathbb{R}^+)}
$$
holds for $1<p<\infty$, and the constant $\frac{p}{p-1}$ is the best. 

For the higher-dimensional case $n\geq2$, generally speaking, there exist two different definitions. One is the rectangle averaging operator defined by
$$
\mathfrak{H}(f)(x_1,\ldots,x_n)=\frac{1}{x_1\cdots x_n}\int_{0
}^{x_1}\cdots\int_{0}^{x_n}f(t_1,\ldots,t_n)d t_1 \cdots d t_n,
$$
where the function $f$ is a nonnegative measurable function on $G=(0,\infty)^n$ and $x_i>0$, for $i=1,2,\ldots,m$, the norm of $\|\mathfrak{H}\|_{L^p\rightarrow L^p}$ is $(\frac{p}{p-1})^n$ and obviously depends on the dimension of the space.

Another definition is the spherical averaging operator, which  was introduced by Christ and Grafakos in \cite{Christ} as follows
$$
\mathcal{H}(f)(x)=\frac{1}{|B(0,|x|)|}\int_{|y|<|x|}f(y)dy, x\in\mathbb{R}^n\backslash\{0\},
$$
where $f$ is a nonnegative measurable function on $\mathbb{R}^n$. The norm of $\mathcal{H}$ on $L^p(\mathbb{R}^n)$ is still $\frac{p}{p-1}$, and doesn't depend on the dimension of the space.

In \cite{Lu}, Lu et al. defined the Hardy operator on higher-dimensional product spaces. Let us recall their definition.
\begin{defn}
	Let $m\in\mathbb{N}$, $n_i\in\mathbb{N}$, $x_i\in\mathbb{R}^{n_i}$, $1\leq i\leq m$ and $f$ be a nonnegative measurable function on $\mathbb{R}^{n_1}\times\mathbb{R}^{n_2}\times\cdots\times\mathbb{R}^{n_m}$. The Hardy operator on higher-dimensional product spaces is defined by
	\begin{equation}
	\mathcal{H}_m(f)(x)=\left(\prod^m_{i=1}\frac{1}{|B(0,|x_i|)|}\right)\int_{|y_1|<|x_1|}\cdots\int_{|y_m|<|x_m|}f(y_1,\ldots,y_m)dy_1\cdots d y_m,
	\end{equation}
	where $x=(x_1,x_2,\ldots,x_m)\in\mathbb{R}^{n_1}\times\cdots\times\mathbb{R}^{n_m}$ with $\prod^m_{i=1}|x_i|\neq0$.
\end{defn}
Next, we will study the sharp bounds for Hardy type operators in product spaces on Heisenberg group. Allow us to introduce some basic knowledge about the Heisenberg group which will be used in the following.

The Heisenberg group $\mathbb{H}^n$ is non-commutative nilpotent Lie group, with the underlying manifold $\mathbb{R}^{2n+1}$ and the group law.

Let
$$
x=(x_1,\ldots,x_{2n},x_{2n+1}),y=(y_1,\ldots,y_{2n},y_{2n+1}),
$$
then we have
$$
x \circ y=\left(x_1+y_1, \ldots, x_{2 n}+y_{2 n}, x_{2 n+1}+y_{2 n+1}+2 \sum_{j=1}^n(y_j x_{n+j}-x_j y_{n+j})\right).
$$
By definition, we can see that the identity element on $\mathbb{H}^n$ is $0 \in \mathbb{R}^{2 n+1}$, while the element $x^{-1}$ inverse to $x$ is $-x$. The corresponding Lie algebra is generated by the left-invariant vector fields
$$
\begin{gathered}
	X_j=\frac{\partial}{\partial x_j}+2 x_{n+j} \frac{\partial}{\partial x_{2 n+1}}, \quad j=1,2, \ldots, n, \\
	X_{n+j}=\frac{\partial}{\partial x_{n+j}}-2 x_j \frac{\partial}{\partial x_{2 n+1}}, \quad j=1,2, \ldots, n, \\
	X_{2 n+1}=\frac{\partial}{\partial x_{2 n+1}} .
\end{gathered}
$$
The only non-trivial commutator relation is
$$
\left[X_j, X_{n+j}\right]=-4 X_{2 n+1}, \quad j=1,2 \ldots, n.
$$
Note that  Heisenberg group $\mathbb{H}^n$ is a homogeneous group with dilations
$$
\delta_r(x_1, x_2, \ldots, x_{2 n}, x_{2 n+1})=(r x_1, r x_2, \ldots, r x_{2 n}, r^2 x_{2 n+1}), \quad r>0.
$$

The Haar measure on $\mathbb{H}^n$ coincides with the usual Lebesgue measure on $\mathbb{R}^{2n+1}$. Denoting any measurable set $E \subset \mathbb{H}^n$ by $|E|$, then we obtain
$$
|\delta_r(E)|=r^Q|E|, d(\delta_r x)=r_Q d x,
$$
where $Q=2n+2$ is called the homogeneous dimension of $\mathbb{H}^n$.

The Heisenberg distance derived from the norm
$$
|x|_h=\left[\left(\sum_{i=1}^{2 n} x_i^2\right)^2+x_{2 n+1}^2\right]^{1 / 4},
$$
where $x=(x_1,x_2,\ldots,x_{2n},x_{2n+1})$ is given by
$$
d(p, q)=d(q^{-1} p, 0)=|q^{-1} p|_h.
$$
This distance $d$ is left-invariant in the sense that $d(p,q)$ remains unchanged when $p$ and $q$ are both left-translated by some fixed vector on $\mathbb{H}^n$. Besides, $d$ satisfies the triangular inequality defined by \cite{Kor}
$$
d(p, q) \leq d(p, x)+d(x, q), \quad p, x, q \in \mathbb{H}^n.
$$
For $r>0$ and $x\in\mathbb{H}^n$, the ball and sphere with center $x$ and radius $r$ on $\mathbb{H}^n$ are given by
$$
B(x, r)=\{y \in \mathbb{H}^n: d(x, y)<r\}
$$
and
$$
S(x, r)=\{y \in \mathbb{H}^n: d(x, y)=r\}.
$$
Then we have
$$
|B(x, r)|=|B(0, r)|=\Omega_Q r^Q,
$$
where
$$
\Omega_Q=\frac{2 \pi^{n+\frac{1}{2}} \Gamma(n / 2)}{(n+1) \Gamma(n) \Gamma((n+1) / 2)}
$$
denote the volume of the unit ball $B(0,1)$ on $\mathbb{H}^n$  that is $\omega_Q=Q \Omega_Q$ (see \cite{CT}). For more details about Heisenberg group can be refer to \cite{GB} and \cite{ST}.

In recent years, many operators in harmonic analysis have been proved to be bounded on Heisenberg group. For instance, Wu and Fu \cite{Wu} studied the sharp bound for $n$-dimensional Hardy operator in Lebesgue space on Heisenberg group. Guo et al. \cite{Guo} studied the Hausdorff operator on Heisenberg group. Hang et al. have also conducted many related studies(see \cite{Hang}, \cite{LiGu} and \cite{LiCen}). 

Now, we provide the definition of Hardy-type operator in higher-dimensional spaces on Heisenberg group.
\begin{defn}
		Let $m\in\mathbb{N}$, $n_i\in\mathbb{N}$, $x_i\in\mathbb{H}^{n_i}$, $1\leq i\leq m$ and $f$ be a nonnegative measurable function on $\mathbb{H}^{n_1}\times\mathbb{H}^{n_2}\times\cdots\times\mathbb{H}^{n_m}$. The Hardy-type operator on Heisenberg group is defined by
	\begin{equation}\label{eq1}
	\mathcal{P}_m(f)(x)=\left(\prod^m_{i=1}\frac{1}{|B(0,|x_i|_h)|}\right)\int_{|y_1|_h<|x_1|_h}\cdots\int_{|y_m|_h<|x_m|_h}f(y_1,\ldots,y_m)dy_1\cdots d y_m,
	\end{equation}
	where $x=(x_1,x_2,\ldots,x_m)\in\mathbb{H}^{n_1}\times\cdots\times\mathbb{H}^{n_m}$ with $\prod^m_{i=1}|x_i|_h\neq0$.
\end{defn}

In \cite{Chu}, Chu et al. defined the $n$-dimensional weighted Hardy operator on Heisenberg group $\mathcal{H}_{hw}$ and $n$-dimensional weighted Ces$\mathrm{\grave{a}}$ro operator on Heisenberg group $\mathcal{H}^*_{hw}$. Let us recall their definition.
\begin{defn}
	Let $w:[0,1]\rightarrow[0,\infty)$ be a function, for a measurable function $f$ on $\mathbb{H}^n$. The $n$-dimensional weighted Hardy operator on Heisenberg group $\mathcal{H}_{hw}$ is defined by 
	\begin{equation}
	\mathcal{H}_{hw}:=\int_{0}^{1}f(\delta_t x)w(t)dt,\quad x\in\mathbb{H}^n. 
	\end{equation}
\end{defn}   
\begin{defn}
	For a nonnegative function $w:[0,1]\rightarrow(0,\infty)$. For a measurable complex-valued function $f$ on $\mathbb{H}^n$, the $n$-dimensional weighted  Ces$\grave{a}$ro operator is defined by 
	\begin{equation}
	\mathcal{H}^*_{hw}:=\int_{0}^{1}\frac{f(\delta_{1/t}x)}{t^Q}w(t)dt, \quad x\in\mathbb{H}^n,
	\end{equation}
	which satisfies
	$$
	\int_{\mathbb{H}^n}f(x)(\mathcal{H}_{hw}g)(x)dx=\int_{\mathbb{H}^n}g(x)(\mathcal{H}^*_{hw})(x)dx,
	$$
	where $f\in L^p(\mathbb{H}^n)$, $g\in L^q(\mathbb{H}^n)$,  $1<p<\infty$, $q=p/(p-1)$, $\mathcal{H}_{hw}$ is bounded on $L^p(\mathbb{H}^n)$, and $\mathcal{H}^*_{hw}$ is bounded on $L^q(\mathbb{H}^n)$.  
\end{defn}
Based on the above definitions, we will provide the definition of weighted Hardy operator and weighted Ces$\mathrm{\grave{a}}$ro operator in product spaces on Heisenberg group.
\begin{defn}
Let $m\in\mathbb{N}$, $n_i\in\mathbb{N}$, $x_i\in\mathbb{H}^{n_i}$, $1\leq i\leq m$ and $f$ be a nonnegative measurable function on $\mathbb{H}^{n_1}\times\mathbb{H}^{n_2}\times\cdots\times\mathbb{H}^{n_m}$.The weighted Hardy operator in product space on Heisenberg group is defined by
\begin{equation}
\mathcal{P}_{\varphi,m}(f)(x)=\int_0^1\cdots\int_0^1f(\delta_{t_1}x_1,\ldots,\delta_{t_m}x_m)\varphi(t_1,\ldots,t_m)d t_1\cdots d t_m,
\end{equation}
where $\varphi$ is a nonnegative measurable function on  $\overbrace{[0,1]\times\cdots\times [0,1] }^m$ and  $x=(x_1,x_2,\ldots,x_m)\in\mathbb{H}^{n_1}\times\mathbb{H}^{n_2}\times\cdots\times\mathbb{H}^{n_m}$.
\end{defn}
\begin{defn}
Let $m\in\mathbb{N}$, $n_i\in\mathbb{N}$, $x_i\in\mathbb{H}^{n_i}$, $1\leq i\leq m$ and $f$ be a nonnegative measurable function on $\mathbb{H}^{n_1}\times\mathbb{H}^{n_2}\times\cdots\times\mathbb{H}^{n_m}$.The weighted Ces$\grave{a}$ro operator in product space on Heisenberg group is defined by
\begin{equation}
\mathcal{P}^*_{\varphi,m}(f)(x)=\int_0^1\cdots\int_0^1\frac{f(\delta_{t_1^{-1}}x_1,\ldots,\delta_{t_m^{-1}}x_m)\varphi(t_1,\ldots,t_m)}{|t_1|_h^Q \cdots |t_m|_h^Q}d t_1\cdots d t_m,
\end{equation}
where $\varphi$ is a nonnegative measurable function on  $\overbrace{[0,1]\times\cdots\times [0,1] }^m$ and  $x=(x_1,x_2,\ldots,x_m)\in\mathbb{H}^{n_1}\times\mathbb{H}^{n_2}\times\cdots\times\mathbb{H}^{n_m}$.
\end{defn}
Next, we will give our main results.
\section{Sharp bound for the Hardy-type operator}
In this section, we will give the sharp bound for Hardy type operator in higher-dimensional Lebesgue space on Heisenberg group.
\begin{thm}\label{thm1}
	Let $1<p<\infty$, $m\in\mathbb{N}$, $n_i\in\mathbb{N}$, $x_i\in\mathbb{H}^{n_i}$, 1$\leq i\leq m$ and $f\in L^p(\mathbb{H}^{n_1}\times\mathbb{H}^{n_2}\times\cdots\times\mathbb{H}^{n_m})$. Then the Hardy-type operator $\mathcal{P}_m$ defined in (\ref{eq1}) is bounded on $L^p(\mathbb{H}^{n_1}\times\mathbb{H}^{n_2}\times\cdots\times\mathbb{H}^{n_m})$ and the norm of $\mathcal{P}_m$ can be obtained as follows
	$$
	\|\mathcal{P}_m\|_{L^p(\mathbb{H}^{n_1}\times\mathbb{H}^{n_2}\times\cdots\times\mathbb{H}^{n_m})\rightarrow L^p(\mathbb{H}^{n_1}\times\mathbb{H}^{n_2}\times\cdots\times\mathbb{H}^{n_m})}=\left(\frac{p}{p-1}\right)^m.
	$$
	\end{thm}
	\begin{proof}[Proof of Theorem \ref{thm1}] 
	We merely give the proof with the case $m=2$ for the sake of clarity in writing, the same is true for the general case $m\geq2$. We set
	$$
	g_f(x_1,x_2)=\frac{1}{\omega_{Q_1}}\frac{1}{\omega_{Q_2}}\int_{|\xi_1|_h=1}\int_{|\xi_2|_h=1}f(\delta_{|x_1|_h}\xi_1,\delta_{|x_2|_h}\xi_2)d\xi_1d\xi_2,\quad x\in\mathbb{H}^{n_i},
	$$
then $g$ is a nonnegative radial function with respect to the variables $x_1,x_2$ respectively. By change of variables, we have
$$
\begin{aligned}
	\mathcal{P}_2(g_f)(x_1,x_2)=&\frac{1}{|B(0,|x_1|_h)|}\frac{1}{|B(0,|x_2|_h)|}\int_{B(0,|x_1|_h)}\int_{B(0,|x_2|_h)}\frac{1}{\omega_{Q_1}}\frac{1}{\omega_{Q_2}}\\
	&\times\int_{|\xi_1|_h=1}\int_{|\xi_2|_h=1}f(\delta_{|y_1|_h}\xi_1,\delta_{|y_2|_h}\xi_2)d\xi_1d\xi_2 d y_1 dy_2\\
	=&\frac{1}{\omega_{Q_1}}\frac{1}{\omega_{Q_2}}\int_{|\xi_1|_h=1}\int_{|\xi_2|_h=1}\frac{1}{|B(0,|x_1|_h)|}\frac{1}{|B(0,|x_2|_h)|}\\
	&\times\int_{B(0,|x_1|_h)}\int_{B(0,|x_2|_h)}f(\delta_{|y_1|_h}\xi_1,\delta_{|y_2|_h}\xi_2)dy_1 dy_2 d\xi_1d\xi_2\\
	=&\frac{1}{\omega_{Q_1}}\frac{1}{\omega_{Q_2}}\int_{|\xi_1|_h=1}\int_{|\xi_2|_h=1}\frac{1}{|B(0,|x_1|_h)|}\frac{1}{|B(0,|x_2|_h)|}\\
	&\times\int_0^{|x_1|_h}\int_0^{|x_2|_h}\int_{|y_1^{'}|_h=1}\int_{|y_2^{'}|_h=1}f(\delta_{r_1}\xi_1,\delta_{r_2}\xi_2)\\
	&\times r_1^{Q_1-1}r_2^{Q_2-1} dy_1^{'}dy_2^{'}dr_1dr_2d\xi_1d\xi_2\\
	=&\int_{|\xi_1|_h=1}\int_{|\xi_2|_h=1}\frac{1}{|B(0,|x_1|_h)|}\frac{1}{|B(0,|x_2|_h)|}\\
	&\times\int_{0}^{|x_1|_h}\int_{0}^{|x_2|_h}f(\delta_{r_1}\xi_1,\delta_{r_2}\xi_2)r_1^{Q_1-1}r_2^{Q_2-1}dr_1 dr_2 d\xi_1d\xi_2\\
	=&\mathcal{P}_2(f)(x_1,x_2).
	\end{aligned}
$$
Using H$\ddot{o}$lder's inequality, we get
$$
\begin{aligned}
\|g_f\|_{L^p(\mathbb{H}^{n_1}\times\mathbb{H}^{n_2})}
=&\left(\int_{\mathbb{H}^{n_1}}\int_{\mathbb{H}^{n_2}}\left|\frac{1}{\omega_{Q_1}}\frac{1}{\omega_{Q_2}}\int_{|\xi_1|_h=1}\int_{|\xi_2|_h=1}f_(\delta_{|x_1|_h}\xi_1,\delta_{|x_2|_h}\xi_2)d\xi_1d\xi_2\right|^pd x_1 d x_2\right)^{1/p}\\
\leq&\frac{1}{\omega_{Q_1}}\frac{1}{\omega_{Q_2}}\left\{\int_{\mathbb{H}^{n_1}}\int_{\mathbb{H}^{n_2}}\left(\int_{|\xi_1|_h=1}\int_{|\xi_2|_h=1}|f(\delta_{|x_1|_h}\xi_1,\delta_{|x_2|_h}\xi_2)|^p d\xi_1 d \xi_2\right)\right.\\
&\times\left.\left(\int_{|\xi_1|_h=1}\int_{|\xi_2|_h=1}d\xi_1d\xi_2\right)^{p/p^{'}}d x_1 d x_2\right\}^{1/p}\\
=&\frac{1}{\omega^{1/p}_{Q_1}}\frac{1}{\omega^{1/p}_{Q_2}}\left\{ \int_{0}^{+\infty}\int_{0}^{+\infty}\int_{|x^{'}_1|_h}\int_{|x^{'}_2|_h}\left(\int_{|\xi_1|_h=1}\int_{|\xi_2|_h=1}|f(\delta_{r_1}\xi_1,\delta_{r_2}\xi_2)|^pd\xi_1 d\xi_2\right)\right.\\
&\times \left. r_1^{Q_1-1} r_2^{Q_2-1}dx_1^{'}d x_2^{'}dr_1dr_2 \right\}^{1/p}\\
=&\|f\|_{L^p(\mathbb{H}^{n_1}\times\mathbb{H}^{n_2})}.
\end{aligned}
$$
Thus, we have
$$
\frac{\|\mathcal{P}_2(f)\|_{L^p(\mathbb{H}^{n_1}\times\mathbb{H}^{n_2})}}{\|f\|_{L^p(\mathbb{H}^{n_1}\times\mathbb{H}^{n_2})}}\leq\frac{\|\mathcal{P}_2(g_f)\|_{L^p(\mathbb{H}^{n_1}\times\mathbb{H}^{n_2})}}{\|g_f\|_{L^p(\mathbb{H}^{n_1}\times\mathbb{H}^{n_2})}}.
$$
This implies the operator $\mathcal{P}_2$ and its restriction to radial function have same norm in $L^p(\mathbb{H}^{n_1}\times\mathbb{H}^{n_2})$. Without loss of generality, we assume that $f$ is a radial function in the rest of the proof.

By changing variables, we have that
$$
\mathcal{P}_2(f)=\frac{1}{\Omega_{Q_1}}\frac{1}{\Omega_{Q_2}}\int_{|z_1|_h<1}\int_{|z_2|_h<1}f(\delta_{|x_1|_h}z_1,\delta_{|x_2|_h}z_2)dz_1dz_2.
$$
Using Minkowski's inequality, we can get
$$
\begin{aligned}
&\|\mathcal{P}_2(f)\|_{L^p(\mathbb{H}^{n_1}\times\mathbb{H}^{n_2})}\\
=&\left(\int_{\mathbb{H}^{n_1}}\int_{\mathbb{H}^{n_2}}\left|\frac{1}{\Omega_{Q_1}}\frac{1}{\Omega_{Q_2}}\int_{|z_1|_h<1}\int_{|z_2|_h<1}f(\delta_{|x_1|_h}z_1,\delta_{|x_2|_h}z_2)dz_1dz_2\right|^pdx_1 dx_2\right)^{1/p}\\
\leq&\frac{1}{\Omega_{Q_1}}\frac{1}{\Omega_{Q_2}}\int_{|z_1|_h<1}\int_{|z_2|_h<1}\left(\int_{\mathbb{H}^{n_1}}\int_{\mathbb{H}^{n_2}}|f(\delta_{|z_1|_h}x_1,\delta_{|z_2|_h}x_2)|^pdx_1dx_2\right)^{1/p}dz_1dz_2\\
=&\frac{1}{\Omega_{Q_1}}\frac{1}{\Omega_{Q_2}}\int_{|z_1|_h<1}\int_{|z_2|_h<1}\left(\int_{\mathbb{H}^{n_1}}\int_{\mathbb{H}^{n_2}}|f(x_1,x_2)|^pdx_1dx_2\right)^{1/p}\\
&\times|z_1|_h^{-Q_1/p}|z_2|_h^{-Q_2/p}dz_1dz_2\\
=&\frac{1}{\Omega_{Q_1}}\frac{1}{\Omega_{Q_2}}\int_{|z_1|_h<1}\int_{|z_2|_h<1}|z_1|_h^{-Q_1/p}|z_2|_h^{-Q_2/p}dz_1dz_2\|f\|_{L^p(\mathbb{H}^{n_1}\times\mathbb{H}^{n_2})}\\
=&\left(\frac{p}{p-1}\right)^2\|f\|_{L^p(\mathbb{H}^{n_1}\times{H}^{n_2})}.
\end{aligned}
$$	
Therefore, it implies that
$$
\|\mathcal{P}_2\|_{L^p(\mathbb{H}^{n_1}\times\mathbb{H}^{n_2})\rightarrow L^p(\mathbb{H}^{n_1}\times\mathbb{H}^{n_2})}\leq\left(\frac{p}{p-1}\right)^2.
$$
On the other hand, for 
$$
0<\varepsilon<\min\left\{1,\frac{(p-1)Q_1}{p},\frac{(p-1)Q_2}{p}\right\},
$$
taking
$$
f_{\varepsilon}(x_1,x_2)=|x_1|_h^{-\frac{Q_1}{p}+\varepsilon}|x_2|_h^{-\frac{Q_2}{p}+\varepsilon}\chi_{\{|x_1|_h<1,|x_2|_h<1\}}(x_1,x_2),
$$
then we have
$$
\|f_\varepsilon\|^p_{L^p(\mathbb{H}^{n_1}\times\mathbb{H}^{n_2})}=\frac{\omega_{Q_1}\omega_{Q_2}}{\varepsilon^2p^2}.
$$
Let us rewrite $\mathcal{P}_2(f_{\varepsilon})$ as follows
$$
\begin{aligned}
\mathcal{P}_2(f_{\varepsilon})=&\frac{1}{|B(0,|x_1|_h)||B(0,|x_2|_h)|} \int_{|y_1|_h<|x_1|_h} \int_{|y_2|_h<|x_2|_h} f_{\varepsilon}(y_1, y_2) d y_1 d y_2\\
=&\frac{1}{|B(0,1)||B(0,1)|} \int_{|z_1|_h<1} \int_{|z_2|_h<1} f_{\varepsilon}(\delta_{|x_1|_h}z_1, \delta_{|x_2|_h}z_2) d z_1 d z_2\\
=&\frac{|x_1|_h^{-\frac{Q_1}{p}+\varepsilon}|x_2|_h^{-\frac{Q_2}{p}+\varepsilon}}{|B(0,1)||B(0,1)|}\\
&\times\int_{\{|z_1|_h<1,|z_1|_h<|x_1|_h^{-1}\}}\int_{\{|z_2|_h<1,|z_2|_h<|x_2|_h^{-1}\}}|z_1|_h^{-\frac{Q_1}{p}+\varepsilon}|z_2|_h^{-\frac{Q_2}{p}+\varepsilon}dz_1dz_2.
\end{aligned}
$$
Consequently, we have
$$
\begin{aligned}
&\|\mathcal{P}_2(f_{\varepsilon})\|^p_{L^p(\mathbb{H}^{n_1}\times\mathbb{H}^{n_2})}\\
=&\frac{1}{(\Omega_{Q_1}\Omega_{Q_2})^p}\int_{\mathbb{H}^{n_1}}\int_{\mathbb{H}^{n_2}}\left|\int_{\{|z_1|_h<1,|z_1|_h<|x_1|_h^{-1}\}}\int_{\{|z_2|_h<1,|z_2|_h<|x_2|_h^{-1}\}}|z_1|_h^{-\frac{Q_1}{p}+\varepsilon}\right.\\
&\times\left.|z_2|_h^{-\frac{Q_2}{p}+\varepsilon}dz_1dz_2\right|^p|x_1|_h^{p\varepsilon-Q_1}|x_2|_h^{p\varepsilon-Q_2}dx_1dx_2\\
\geq&\frac{1}{(\Omega_{Q_1}\Omega_{Q_2})^p}\int_{|x_1|_h<1}\int_{|x_2|_h<1}\left|\int_{|z_1|_h<1}\int_{|z_2|_h<1}|z_1|_h^{-\frac{Q_1}{p}+\varepsilon}\right.\\
&\times\left.|z_2|_h^{-\frac{Q_2}{p}+\varepsilon}dz_1dz_2\right|^p|x_1|_h^{p\varepsilon-Q_1}|x_2|_h^{p\varepsilon-Q_2}dx_1dx_2\\
=&\frac{\omega_{Q_1}\omega_{Q_2}}{\varepsilon^2p^2}\frac{1}{(\Omega_{Q_1}\Omega_{Q_2})^p}\left(\int_{|z_1|_h<1}\int_{|z_2|_h<1}|z_1|_h^{-\frac{Q_1}{p}+\varepsilon}|z_2|_h^{-\frac{Q_2}{p}+\varepsilon}dz_1dz_2\right)^p\\
=&\left(\frac{p}{p-1+{r\varepsilon/Q_1}}\frac{p}{p-1+{r\varepsilon/Q_2}}\right)^p\|f_{\varepsilon}\|_{L^p(\mathbb{H}^{n_1}\times\mathbb{H}^{n_2})}.
\end{aligned}
$$ 
Thus, we have
$$
\|\mathcal{P}_2\|_{L^p(\mathbb{H}^{n_1}\times\mathbb{H}^{n_2})\rightarrow L^p(\mathbb{H}^{n_1}\times\mathbb{H}^{n_2})}\geq \frac{p}{p-1+{r\varepsilon/Q_1}}\frac{p}{p-1+{r\varepsilon/Q_2}}.
$$
Consequently, using the definition of the norm of an operator and letting $\varepsilon\rightarrow0$, we conclude that
$$
\|\mathcal{P}_2\|_{L^p(\mathbb{H}^{n_1}\times\mathbb{H}^{n_2})\rightarrow L^p(\mathbb{H}^{n_1}\times\mathbb{H}^{n_2})}\geq\left(\frac{p}{p-1}\right)^2.
$$
This finishes the proof of the theorem. 
\end{proof}
\section{Boundedness for the weighted Hardy operator and weighted Ces$\grave{a}$ro operator}
In this section, we will study the boundedness of weighted Hardy-type operator and give the necessary and sufficient conditions for the boundedness.
\begin{thm}\label{thm2}
Let $m\in\mathbb{N}$, $n_i\in\mathbb{N}$, $x_i\in\mathbb{H}^{n_i}$, $1\leq i\leq m$, $\varphi$ is a nonnegative measurable function on  $\overbrace{[0,1]\times\cdots\times [0,1] }^m$. If $f\in L^p(\mathbb{H}^{n_1}\times\mathbb{H}^{n_2}\times\cdots\times\mathbb{H}^{n_m})$, then the weighted Hardy operator in product space on Heisenberg group $\mathcal{P}_{\varphi,m}$ is bounded on $L^p(\mathbb{H}^{n_1}\times\mathbb{H}^{n_2}\times\cdots\times\mathbb{H}^{n_m})$ if and only if 
$$
\int_{0}^{1}\cdots\int_0^1 |t_1|_h^{-Q_1/p}\cdots|t_m|_h^{-Q_m/p}\varphi(t_1,\ldots,t_m)dt_1\cdots dt_m<\infty.
$$
\end{thm} 
\begin{thm}\label{thm3}
	Let $m\in\mathbb{N},n_i\in\mathbb{N},x_i\in\mathbb{H}^{n_i},1\leq i\leq m$,  $\varphi$ is a nonnegative measurable function on  $\overbrace{[0,1]\times\cdots\times [0,1] }^m$. If $f\in L^p(\mathbb{H}^{n_1}\times\mathbb{H}^{n_2}\times\cdots\times\mathbb{H}^{n_m})$, then the weighted Ces$\grave{a}$ro  operator in product space on Heisenberg group $\mathcal{P}^*_{\varphi,m}$ is bounded on $L^p(\mathbb{H}^{n_1}\times\mathbb{H}^{n_2}\times\cdots\times\mathbb{H}^{n_m})$ if and only if 
	$$
	\int_{0}^{1}\cdots\int_0^1 |t_1|_h^{-Q_1(1-1/p)}\cdots|t_m|_h^{-Q_m(1-1/p)}\varphi(t_1,\ldots,t_m)dt_1\cdots dt_m<\infty.
	$$
\end{thm} 
The proof methods for Theorem \ref{thm2} and Theorem \ref{thm3} are the same. We only give the proof of Theorem \ref{thm2}.
\begin{proof}[Proof of Theorem \ref{thm2}]
We merely give the proof with the case $m=2$ for the sake of clarity in writing, the same is true for the general case $m\geq2$.

Since the case $p=\infty$ is trivial, it suffices to consider the case $1\leq p<\infty$.

Using Minkowski's inequality and the change of variables, we have
$$
\begin{aligned}
\|\mathcal{P}_{\varphi,2}\|_{L^p(\mathbb{H}^{n_1}\times\mathbb{H}^{n_2})}=&\left(\int_{\mathbb{H}^{n_1}}\int_{\mathbb{H}^{n_2}}\left|\int_{0}^{1}\int_{0}^{1}f(\delta_{t_1}x_1,\delta_{t_2}x_2)\varphi(t_1,t_2)d t_1 dt_2\right|^pdx_1dx_2\right)^{1/p}\\
\leq&\int_{0}^{1}\int_{0}^{1}\left(\int_{\mathbb{H}^{n_1}}\int_{\mathbb{H}^{n_2}}|f(\delta_{t_1}x_1,\delta_{t_2}x_2)|^pdx_1dx_2\right)^{1/p}\varphi(t_1,t_2)dt_1dt_2\\
=&\int_{0}^{1}\int_{0}^{1}\left(\int_{\mathbb{H}^{n_1}}\int_{\mathbb{H}^{n_2}}|f(x_1,x_2)|^pdx_1dx_2\right)^{1/p}\\
&\times|t_1|_h^{-Q_1/p}|t_2|_h^{-Q_2/p}\varphi(t_1,t_2)dt_1dt_2\\
=&\int_{0}^{1}\int_{0}^{1}|t_1|_h^{-Q_1/p}|t_2|_h^{-Q_2/p}\varphi(t_1,t_2)dt_1dt_2\|f\|_{L^p(\mathbb{H}^{n_1}\times\mathbb{H}^{n_2})}.
\end{aligned}
$$
Therefore, we have
\begin{equation}\label{main_2}
\|\mathcal{P}_{\varphi,2}\|_{L^p(\mathbb{H}^{n_1}\times\mathbb{H}^{n_2})\rightarrow L^p(\mathbb{H}^{n_1}\times\mathbb{H}^{n_2})}\leq\int_{0}^{1}\int_{0}^{1}|t_1|_h^{-Q_1/p}|t_2|_h^{-Q_2/p}\varphi(t_1,t_2)dt_1dt_2.
\end{equation}
Next, taking 
$$
C=\|\mathcal{P}_{\varphi,2}\|_{L^p(\mathbb{H}^{n_1}\times\mathbb{H}^{n_2})\rightarrow L^p(\mathbb{H}^{n_1}\times\mathbb{H}^{n_2})}<\infty
$$
and for $f\in L^p(\mathbb{H}^{n_1}\times\mathbb{H}^{n_2})$, we have
$$
\|\mathcal{P}_{\varphi,2}\|_{L^p(\mathbb{H}^{n_1}\times\mathbb{H}^{n_2})}\leq C\|f\|_{L^p(\mathbb{H}^{n_1}\times\mathbb{H}^{n_2})}.
$$
For $\epsilon>0$, taking 
$$
f_\epsilon(x)= \begin{cases}0, & |x|_h \leq 1, \\ |x|_h^{-\left(\frac{Q}{p}+\epsilon\right)} & |x|_h>1\end{cases},
$$
we obtain
$$
\|f_\epsilon\|^p_{L^p(\mathbb{H}^{n_1}\times\mathbb{H}^{n_2})}=\frac{\omega_{Q_1}}{\epsilon p}\frac{\omega_{Q_2}}{\epsilon p}
$$
and
$$
\mathcal{P}_{\varphi,2}(f_\epsilon)(x)= \begin{cases}0, & |x|_h \leq 1, \\ |x_1|_h^{-{\frac{Q_1}{p}}-\epsilon}|x_2|_h^{-{\frac{Q_2}{p}}-\epsilon} \int_{|x_1|_h^{-1}}^1\int_{|x_2|_h^{-1}}^1|t_1|_h^{-\frac{Q_1}{p}-\epsilon}|t_1|_h^{-\frac{Q_2}{p}-\epsilon}\varphi(t_1,t_2) d t_1 dt_2, & |x|_h>1\end{cases}.
$$
So, we have  
$$
\begin{aligned}
	C^p\|f_\epsilon\|^p_{L^p(\mathbb{H}^{n_1}\times\mathbb{H}^{n_2})}\geq& \|\mathcal{P}_{\varphi,2}\|^p_{L^p(\mathbb{H}^{n_1}\times\mathbb{H}^{n_2})}\\
	=&\int_{|x_1|_h>1}\int_{|x_2|_h>1}\left(|x_1|_h^{-Q_1/p-\epsilon}|x_2|_h^{-Q_2/p-\epsilon}\right.\\
	&\times\left.\int_{|x_1|_h^{-1}}^1\int_{|x_2|_h^{-1}}^1|t_1|_h^{-\frac{Q_1}{p}-\epsilon}|t_2|_h^{-\frac{Q_2}{p}-\epsilon}\varphi(t_1,t_2) d t_1 dt_2\right)^p dx_1 dx_2\\
	\geq&\int_{|x_1|_h>\epsilon^{-1}}\int_{|x_2|_h>\epsilon^{-1}}\left(|x_1|_h^{-Q_1/p-\epsilon}|x_2|_h^{-Q_2/p-\epsilon}\right.\\
	&\times\left.\int_{\epsilon}^1\int_{\epsilon}^1|t_1|_h^{-\frac{Q_1}{p}-\epsilon}|t_2|_h^{-\frac{Q_2}{p}-\epsilon}\varphi(t_1,t_2) d t_1 dt_2\right)^p dx_1 dx_2\\
	=&\int_{|x_1|_h>\epsilon^{-1}}\int_{|x_2|_h>\epsilon^{-1}}|x_1|_h^{-Q_1-\epsilon p}|x_2|_h^{-Q_2-\epsilon p}d x_1 dx_2\\
	&\times\left(\int_{\epsilon}^1\int_{\epsilon}^1|t_1|_h^{-\frac{Q_1}{p}-\epsilon}|t_2|_h^{-\frac{Q_2}{p}-\epsilon}\varphi(t_1,t_2) d t_1 dt_2\right)^p \\
	=&\left(\int_{\epsilon}^1\int_{\epsilon}^1|t_1|_h^{-\frac{Q_1}{p}-\epsilon}|t_2|_h^{-\frac{Q_2}{p}-\epsilon}\varphi(t_1,t_2) d t_1 dt_2\right)^p\|f_\epsilon\|^p_{L^p(\mathbb{H}^{n_1}\times\mathbb{H}^{n_2})}.
\end{aligned}
$$
This implies that 
$$
\|\mathcal{P}_{\varphi,2}\|_{L^p(\mathbb{H}^{n_1}\times\mathbb{H}^{n_2})\rightarrow L^p(\mathbb{H}^{n_1}\times\mathbb{H}^{n_2})}\geq\int_{\epsilon}^1\int_{\epsilon}^1|t_1|_h^{-\frac{Q_1}{p}-\epsilon}|t_2|_h^{-\frac{Q_2}{p}-\epsilon}\varphi(t_1,t_2) d t_1 dt_2.
$$
Since $\epsilon\rightarrow0$, we can obtain that
\begin{equation}\label{main_3}
\|\mathcal{P}_{\varphi,2}\|_{L^p(\mathbb{H}^{n_1}\times\mathbb{H}^{n_2})\rightarrow L^p(\mathbb{H}^{n_1}\times\mathbb{H}^{n_2})}\geq\int_{0}^1\int_{0}^1|t_1|_h^{-\frac{Q_1}{p}}|t_2|_h^{-\frac{Q_2}{p}}\varphi(t_1,t_2) d t_1 dt_2.
\end{equation}
Combining (\ref{main_2}) and (\ref{main_3}), we have finished the proof. 
\end{proof}
\par
\subsection*{Acknowledgements}
This work was supported by National Natural Science Foundation of  China (Grant No. 12271232) and Shandong Jianzhu University Foundation (Grant No. X20075Z0101).

\begin{flushleft}

	\vspace{0.3cm}\textsc{Zhongci Hang\\School of Science\\Shandong Jianzhu University \\Jinan, 250000\\P. R. China}
	
	\emph{E-mail address}: \textsf{babysbreath4fc4@163.com}

	\vspace{0.3cm}\textsc{Wenfeng Liu\\School of Science\\Shandong Jianzhu University \\Jinan, 250000\\P. R. China}
	
	\emph{E-mail address}: \textsf{19546103663@163.com}
	
	\vspace{0.3cm}\textsc{Xiang Li\\School of Science\\Shandong Jianzhu University\\Jinan, 250000\\P. R. China}
	
	\emph{E-mail address}: \textsf{lixiang162@mails.ucas.ac.cn}

	\vspace{0.3cm}\textsc{Dunyan Yan\\School of Mathematical Sciences\\University of Chinese Academy of Sciences\\Beijing, 100049\\P. R. China}
	
	\emph{E-mail address}: \textsf{ydunyan@ucas.ac.cn}
	
\end{flushleft}


\begin{thebibliography}{99}
	
	
 \bibitem{Christ}M. Christ, L. Grafakos, Best constants for two nonconvolution  inequalities, Proc.Amer.Math.Soc., 1995, 123 (6): 1687--1693.
 
 \bibitem{Chu} J. Chu, Z. Fu and Q. Wu, $L^p$ and BMO bounds for weighted Hardy operators on the Heisenberg group, J Ineq Appl., 2016, Article 282: 12 pp.
 
 \bibitem{CT} T. Coulhon T, D. M\"{u}ller  and J. Zienkiewicz, About Riesz transforms on the Heisenberg groups, Math. Ann., 1996, 305 (2): 369--379.
	
 \bibitem{GB} G.B. Folland and E.M. Stein, Hardy spaces on homogeneous groups, Princeton, N. J. Princeton University Press, 1982.
 
 \bibitem{Guo} J. Guo, L. Sun and F. Zhao, Hausdorff operators on the Heisenberg group,  Acta Math. Sci., 2015, 31 (11): 1703--1714.

 \bibitem{Hang} Z. Hang, X. Li and D. Yan, Mixed radial-angular bounds for Hardy-type operators on Heisenberg groups, arXiv.2304.11792.
 
 \bibitem{Hardy} G.H. Hardy, Note on a theorem of Hilbert, Math. Z., 1920, 6 (3): 314--317.
 
\bibitem{Kor} A. Kor\"{a}anyi, H.M. Reimann, Quasiconformal mappings on the Heisenberg group, Invent. Math., 1985, 80 (2): 309--338.
  
\bibitem{LiCen} X. Li, X. Cen, Z. Fu and Z. Hang, Boundedness of the multilinear integral operators on Heisenberg group, arXiv. 2304.07767.	

\bibitem{LiGu} X. Li, Z. Gu, D. Yan and Z. Hang, Sharp bounds for a class of integral operators in weighted--type spaces on Heisenberg group, arXiv.2304.08755.
 
 \bibitem{Lu} S. Lu, D. Yan, F. Zhao, Sharp bounds for Hardy type operators on higher-dimensional product spaces, J. Inequal. Appl., 2013 (1): 148--159.
 

\bibitem{ST} S. Thangavelu, Harmonic analysis on the Heisenberg group, Progress in Mathematics, vol. 159, Boston, MA: Birkhauser Boston, 1998.


\bibitem{Wu} Q. Wu and Z. Fu, Sharp estimates for Hardy operators on Heisenberg group, Front. Math.
China, 2016, 11 (1): 155--172.









\end{thebibliography}
\end{document}